\documentclass[a4paper,11pt]{article}
\usepackage{a4wide}
\usepackage{amssymb}
\usepackage{amsmath}
\usepackage{mathtools}
\usepackage{amsthm}
\usepackage[english]{babel}
\usepackage{lmodern}
\usepackage[final,nopatch=footnote]{microtype}
\usepackage[mathscr]{euscript}
\usepackage[dvipsnames]{xcolor}
\usepackage{todonotes}
\usepackage{appendix}
\usepackage{bm}
\usepackage{cancel}
\usepackage{upgreek}
\usepackage[normalem]{ulem}
\usepackage{zref-clever}
\usepackage[colorlinks=true]{hyperref}
\usepackage{enumerate}

\numberwithin{equation}{section}

\let\oldnewtheorem\newtheorem

\RenewDocumentCommand{\newtheorem}{momo}{%
  \IfValueTF{#2}{%
    \AddToHook{env/#1/begin}{%
      \zcsetup{countertype={#2=#1}}}%
      \zcRefTypeSetup{#1}{
Name-sg = #3 ,%
      }%
    \oldnewtheorem{#1}[#2]{#3}%
  }{%
    \AddToHook{env/#1/begin}{
      \zcsetup{countertype={#1=#1}}}%
    \zcRefTypeSetup{#1}{
Name-sg = #3 ,%
      }%
    \IfValueTF{#4}{%
      \oldnewtheorem{#1}{#3}[#4]%
    }{%
      \oldnewtheorem{#1}{#3}%
    }%
  }%
}

\newcommand{\cref}[1]{\zcref{#1}}
\newcommand{\Cref}[1]{\zcref[S]{#1}}

\theoremstyle{definition}
\newtheorem{definition}{Definition}[section]
\newtheorem{remark}[definition]{Remark}

\theoremstyle{plain}

\newtheorem{theorem}[definition]{Theorem}
\newtheorem{proposition}[definition]{Proposition}
\newtheorem{lemma}[definition]{Lemma}
\newtheorem{corollary}[definition]{Corollary}

\makeatletter
\let\save@mathaccent\mathaccent
\newcommand*\if@single[3]{%
\setbox0\hbox{${\mathaccent"0362{#1}}^H$}%
\setbox2\hbox{${\mathaccent"0362{\kern0pt#1}}^H$}%
\ifdim\ht0=\ht2 #3\else #2\fi
}
\newcommand*\rel@kern[1]{\kern#1\dimexpr\macc@kerna}
\newcommand*\widebar[1]{\@ifnextchar^{{\wide@bar{#1}{0}}}{\wide@bar{#1}{1}}}
\newcommand*\wide@bar[2]{\if@single{#1}{\wide@bar@{#1}{#2}{1}}{\wide@bar@{#1}{#2}{2}}}
\newcommand*\wide@bar@[3]{%
\begingroup
\def\mathaccent##1##2{%
\let\mathaccent\save@mathaccent
\if#32 \let\macc@nucleus\first@char \fi
\setbox\z@\hbox{$\macc@style{\macc@nucleus}_{}$}%
\setbox\tw@\hbox{$\macc@style{\macc@nucleus}{}_{}$}%
\dimen@\wd\tw@
\advance\dimen@-\wd\z@
\divide\dimen@ 3
\@tempdima\wd\tw@
\advance\@tempdima-\scriptspace
\divide\@tempdima 10
\advance\dimen@-\@tempdima
\ifdim\dimen@>\z@ \dimen@0pt\fi
\rel@kern{0.6}\kern-\dimen@
\if#31
\overline{\rel@kern{-0.6}\kern\dimen@\macc@nucleus\rel@kern{0.4}\kern\dimen@}%
\advance\dimen@0.4\dimexpr\macc@kerna
\let\final@kern#2%
\ifdim\dimen@<\z@ \let\final@kern1\fi
\if\final@kern1 \kern-\dimen@\fi
\else
\overline{\rel@kern{-0.6}\kern\dimen@#1}%
\fi
}%
\macc@depth\@ne
\let\math@bgroup\@empty \let\math@egroup\macc@set@skewchar
\mathsurround\z@ \frozen@everymath{\mathgroup\macc@group\relax}%
\macc@set@skewchar\relax
\let\mathaccentV\macc@nested@a
\if#31
\macc@nested@a\relax111{#1}%
\else
\def\gobble@till@marker##1\endmarker{}%
\futurelet\first@char\gobble@till@marker#1\endmarker
\ifcat\noexpand\first@char A\else
\def\first@char{}%
\fi
\macc@nested@a\relax111{\first@char}%
\fi
\endgroup
}
\makeatother

\newcommand{\R}{\mathbf R}

\usepackage[xcolor]{changebar}
\cbcolor{blue}

\newcounter{inlineenum}
\renewcommand{\theinlineenum}{\enumlabelformat{inlineenum}}

\let\epsilon\varepsilon
\let\phi\varphi

\newcommand{\nchi}{{\raise.3ex\hbox{$\chi$}}}

\usepackage[runin]{abstract}

\makeatletter
\let\@fnsymbol\@arabic
\makeatother

\newcommand{\del}{\partial}
\newcommand{\lra}{\longrightarrow}
\newcommand{\ve}{\varepsilon}
\newcommand{\bK}{\mathbf{K}}
\newcommand{\bL}{\mathbf{L}}
\newcommand{\dd}{\mathrm{d}}
\newcommand{\var}{\mathsf{var}}
\newcommand{\cP}{\mathcal{P}}

\begin{document}

\title{Concavity of spacetimes}
\date{\today}
\author{Tobias Beran\thanks{Department of Mathematics, University of Vienna, Oskar-Morgenstern-Platz 1, 1090 Vienna, Austria ({\sf tobias.beran@univie.ac.at}, {\sf darius.eroes@univie.ac.at})}
\and Darius Er\"os\footnotemark[1]
\and Shin-ichi Ohta\thanks{
Department of Mathematics, University of Osaka, Osaka 560-0043, Japan
({\sf s.ohta@math.sci.osaka-u.ac.jp});
RIKEN Center for Advanced Intelligence Project (AIP),
1-4-1 Nihonbashi, Tokyo 103-0027, Japan}
\and Felix Rott\thanks{SISSA, Via Bonomea 265, 34136 Trieste, Italy ({\sf frott@sissa.it})}
}

\maketitle

\begin{abstract}
Motivated by recent breathtaking progress in the synthetic study of Lorentzian geometry,
we investigate the local concavity of time separation functions on Finsler spacetimes as a Lorentzian counterpart to Busemann's convexity in metric geometry.
We show that a Berwald spacetime is locally concave if and only if its flag curvature is nonnegative in timelike directions.
We also give another characterization of nonnegative flag curvature by the convexity of future (or past) capsules, inspired by Krist\'aly--Kozma's result in the positive definite case.
These characterizations are new even for Lorentzian manifolds.
\medskip

\medskip

\end{abstract}

\section{Introduction}

The aim of this article is to develop the synthetic study of Lorentzian geometry.
Synthetic geometry is concerned with geometric spaces without differentiable structures,
such as metric (measure) spaces.
For instance, for a complete Riemannian manifold,
every small triangle drawn with geodesics is thinner than its comparison triangle (with the same sidelengths) in $\R^2$
if and only if the sectional curvature is nonpositive everywhere
(Alexandrov's triangle comparison theorem; see, e.g., \cite[Theorem~II.1A.6]{BH}).
Then, the triangle comparison property can be adopted as a synthetic notion of nonpositive curvature;
a (geodesic) metric space $(X,d)$ is called a \emph{CAT$(0)$-space} if triangles in $X$ are thinner than $\R^2$.
In the same manner, one can define metric spaces of sectional curvature bounded below
(Alexandov spaces; see \cite{BBI, AKP}),
and metric measure spaces of Ricci curvature $\ge K$ and dimension $\le N$
(MCP$(K,N)$-spaces (MCP stands for measure contraction property) and CD$(K,N)$-spaces (CD stands for curvature-dimension condition); see \cite{StI,StII,Omcp,LV}).
These synthetic approaches enable us to understand the influence of curvature in a geometric and intuitive way,
and to investigate metric (measure) spaces arising as limit spaces of Riemannian manifolds,
leading on to fruitful applications in Riemannian geometry and beyond.

Recently, there is a growing interest in synthetic approaches to Lorentzian geometry,
partly motivated by the appearance of less regular spacetimes in general relativity
(as solutions to the Einstein equation).
For time-oriented Lorentzian manifolds (spacetimes),
suitable counterparts to the triangle comparison properties are known to characterize
lower and upper sectional curvature bounds in timelike directions (see \cite{Har82,AB,BKOR}).
Then, as a platform for synthetic Lorentzian geometry,
Lorentzian (pre-)length spaces are introduced \cite{KS},
and the structure of Lorentzian (pre-)length spaces of curvature bounded below or above
in the sense of triangle comparison is the subject of current active research
(see, e.g., \cite{BORS,BS}).
We also refer to \cite{Mc,MS,CM} for a synthetic approach to lower Ricci curvature bounds (timelike curvature-dimension condition TCD$(K,N)$),
to name a few.

For metric spaces, there is also a weaker notion of nonpositive curvature than CAT$(0)$-spaces,
going back to a seminal work of Busemann \cite{Bu},
called the \emph{convexity} or the \emph{Busemann nonpositive curvature}.
Let $(X,d)$ be a \emph{geodesic} space, i.e., any pair $x,y \in X$ of points can be joined by a curve $\gamma\colon [0,1] \lra X$ such that $\gamma(0)=x$, $\gamma(1)=y$ and $d(\gamma(s),\gamma(t))=|t-s|d(x,y)$ for all $s,t \in [0,1]$
(then $\gamma$ is called a \emph{minimizing geodesic}).
We say that $(X,d)$ is \emph{locally convex} if every point $x \in X$ has a neighborhood $U$ such that, for any minimizing geodesics $\gamma_1,\gamma_2\colon [0,1] \lra X$ included in $U$, the function $t \longmapsto d(\gamma_1(t),\gamma_2(t))$ is convex.
If we can take $U=X$, $(X,d)$ is said to be \emph{globally convex}.
For Riemannian manifolds, local convexity is equivalent to having the nonpositive sectional curvature everywhere.

An important feature of Busemann's convexity is that it does not rule out Finsler manifolds.
For example, strictly convex Banach spaces are clearly globally convex,
while only Hilbert spaces are CAT$(0)$ among Banach spaces.
Therefore, it had been an intriguing problem to characterize convex spaces among Finsler manifolds.
It was first shown in \cite{KVK} that a Berwald manifold, a special kind of Finsler manifold,
of nonpositive flag curvature is locally convex
(flag curvature in Finsler geometry corresponds to sectional curvature in Riemannian geometry).
Then, \cite{KK} proved that a Berwald manifold is locally convex only if its flag curvature is nonpositive everywhere.
Finally, it was established in \cite{IL} that a Finsler manifold is locally convex if and only if
it is a Berwald manifold of nonpositive flag curvature.
We refer to \cite{FG} for a recent study of the structure of convex metric spaces,
where we need to deal with the Finsler (non-Riemannian) nature of those spaces (see also \cite{Ke,HY} for related works concerning nonnegative curvature).

As a Lorentzian counterpart to Busemann's convexity,
the concavity of time separation functions on Lorentzian pre-length spaces
of nonnegative curvature in the sense of triangle comparison
was proved in \cite[Section~6]{BKR}.
Moreover, a globalization result for such a concavity was established in \cite{EG} (see Remark~\ref{rm:global}).
Then, in comparison with convexity, it is natural to ask an analog to the above characterization of convex Finsler or Berwald manifolds.
Our main theorem provides an answer to this question in the setting of Berwald spacetimes.

\begin{theorem}[Characterizations of concavity]\label{th:main}
Let $(M,L)$ be a Berwald spacetime.
Then, the following are equivalent.
\begin{enumerate}[{\rm (I)}]
\item\label{it:K>0} $(M,L)$ has nonnegative flag curvature $\bK \ge 0$ in timelike directions, namely $\bK(v,w) \ge 0$ for every pair of linearly independent vectors $v,w \in T_xM$ such that $v$ is future-directed timelike.
\item\label{it:ccv} $(M,L)$ is locally concave.
\item\label{it:c-ccv} $(M,L)$ is locally timelike concave.
\item\label{it:f-cap} $(M,L)$ has convex future capsules.
\item\label{it:p-cap} $(M,L)$ has convex past capsules.
\end{enumerate}
\end{theorem}

See Definition~\ref{df:conc} for the precise definition of local (timelike) concavity.
Having convex future or past capsules is a condition inspired by \cite{KK}; see Definition~\ref{df:capsule}.
Finsler spacetimes generalize usual (Lorentzian) spacetimes
in the same way that Finsler manifolds generalize Riemannian manifolds,
and the Berwald condition is defined in the same way as well.
Finsler and Berwald spacetimes have been studied from both geometric and physical viewpoints
(see Section~\ref{sc:Fspt}).

Theorem~\ref{th:main} is new even in the case of Lorentzian manifolds.
Then, the equivalence between local (timelike) concavity and nonnegative sectional curvature
can be regarded as a direct adaption of the aforementioned result of Busemann \cite{Bu} to the Lorentzian setting.

\begin{corollary}[Lorentzian case]\label{cr:main}
Let $(M,g)$ be a Lorentzian spacetime.
Then, the following are equivalent.
\begin{enumerate}[{\rm (I)}]
\item $(M,g)$ has nonnegative sectional curvature in timelike directions.
\item $(M,g)$ is locally concave.
\item $(M,g)$ is locally timelike concave.
\item $(M,g)$ has convex future capsules.
\item $(M,g)$ has convex past capsules.
\end{enumerate}
\end{corollary}

Our proof of Theorem~\ref{th:main} follows the lines of \cite{KVK,KK}.
Due to a common difficulty in Lorentzian geometry, the noncompactness of the indicatrix in the tangent spaces,
we could not generalize the argument in \cite{IL}.
Thus, it remains an open question whether locally (timelike) concave Finsler spacetimes necessarily satisfy the Berwald condition
(see Subsection~\ref{ssc:nonB} for more details).

This article is organized as follows.
After reviewing the basics of Finsler spacetimes in Section~\ref{sc:Fspt},
we prove the equivalence between concavity and nonnegative flag curvature in Section~\ref{sc:Bconc}.
We then give another characterization of nonnegative flag curvature by convex capsules in Section~\ref{sc:capsule}.
Finally, Section~\ref{sc:problem} is devoted to discussing some open problems on concavity,
compared to the convexity of metric spaces.

\section{Finsler spacetimes}\label{sc:Fspt}

We review the basics of Finsler spacetimes
(see \cite{ON, BEE} for the standard theory of Lorentzian spacetimes).
Our definition of Finsler spacetimes follows Beem's one \cite{Be} (see \cite[\S 2.1]{LMO1} for the relationship with other definitions).
Concerning recent developments in Finsler spacetime geometry,
we refer to \cite{LMO1,LMO2} for singularity and comparison theorems, \cite{LMO3, COO} for timelike splitting theorems, and to \cite{BO} for the timelike curvature-dimension condition.

\subsection{Lorentz--Finsler manifolds}

Throughout the article, let $M$ be a connected $C^{\infty}$-manifold without boundary of dimension $n\ge 2$.

\begin{definition}[Lorentz--Finsler structures]\label{df:LFstr}
A \emph{Lorentz--Finsler structure} on $M$ is a function
$L\colon TM \lra \R$ satisfying the following conditions:
\begin{enumerate}[(1)]
\item $L \in C^{\infty}(TM \setminus \{0\})$;
\item $L(cv)=c^2 L(v)$ for all $v \in TM$ and $c>0$;
\item For any $v \in TM \setminus \{0\}$, the symmetric matrix
\begin{equation}\label{eq:g_ij}
\bigl( g_{ij}(v) \bigr)_{i,j=1}^n
 :=\biggl( \frac{\del^2 L}{\del v^i \del v^j}(v) \biggr)_{i,j=1}^n
\end{equation}
is non-degenerate with signature $(-,+,\ldots,+)$.
\end{enumerate}
Then, we call $(M,L)$ a ($C^{\infty}$-)\emph{Lorentz--Finsler manifold}.
\end{definition}

For each $v \in T_xM \setminus \{0\}$, we can define a Lorentzian metric $g_v$ on $T_xM$ via \eqref{eq:g_ij} as
\[ g_v \Biggl( \sum_{i=1}^n a_i \frac{\del}{\del x^i}\bigg|_x,
 \sum_{j=1}^n b_j \frac{\del}{\del x^j}\bigg|_x \Biggr)
 :=\sum_{i,j=1}^n g_{ij}(v) a_i b_j. \]
Note that we have $g_v(v,v)=2L(v)$ by Euler's homogeneous function theorem.

A tangent vector $v \in TM $ is said to be \emph{timelike} (resp.\ \emph{null})
if $L(v)<0$ (resp.\ $L(v)=0$).
We say that $v$ is \emph{lightlike} if it is null and nonzero, and \emph{causal} if it is timelike or lightlike.
\emph{Spacelike} vectors are $v \in TM$ such that $L(v)>0$ or $v=0$.
Denote by $\Omega'_x \subset T_xM$ the set of timelike vectors.
For causal vectors $v$, we define
\begin{equation}\label{eq:LtoF}
F(v) :=\sqrt{-2L(v)} =\sqrt{-g_v(v,v)}.
\end{equation}

\begin{definition}[Finsler spacetimes]\label{df:Fspt}
If a Lorentz--Finsler manifold $(M,L)$ admits a smooth timelike vector field $X$,
then $(M,L)$ is said to be \emph{time-oriented} (by $X$).
A time-oriented Lorentz--Finsler manifold is called a \emph{Finsler spacetime}.
\end{definition}

In a Finsler spacetime time-oriented by $X$,
a causal vector $v \in T_xM$ is said to be \emph{future-directed}
if it lies in the same connected component of $\overline{\Omega'}\!_x \setminus \{0\}$ as $X(x)$.
We denote by $\Omega_x \subset \Omega'_x$ the set of future-directed timelike vectors,
and define
\[ \Omega :=\bigcup_{x \in M} \Omega_x. \]
A $C^1$-curve in $M$ is said to be \emph{timelike} (resp.\ \emph{causal})
if its tangent vector is always timelike (resp.\ causal).
Henceforth, unless explicitly stated otherwise, all causal curves are assumed to be future-directed (i.e., its tangent vector is future-directed causal).

Next, we introduce the covariant derivative.
Define
\[ \gamma^i_{jk} (v)
 :=\frac{1}{2} \sum_{l=1}^n g^{il}(v)
 \biggl( \frac{\del g_{lk}}{\del x^j} +\frac{\del g_{jl}}{\del x^k} -\frac{\del g_{jk}}{\del x^l} \biggr)(v) \]
for $i,j,k =1,\ldots,n$ and $v \in TM\setminus \{0\}$,
where $(g^{ij}(v))$ is the inverse matrix of $(g_{ij}(v))$;
\[ G^i(v) :=\sum_{j,k=1}^n \gamma^i_{jk}(v) v^j v^k, \qquad
N^i_j(v) :=\frac{1}{2}\frac{\del G^i}{\del v^j}(v) \]
for $v \in TM \setminus \{0\}$, and $G^i(0)=N^i_j(0):=0$.
Moreover, we set
\begin{equation}\label{eq:Gamma}
\Gamma^i_{jk}(v)
 :=\gamma^i_{jk}(v) -\frac{1}{2}\sum_{l,m=1}^n g^{il}(v)
 \biggl( \frac{\del g_{lk}}{\del v^m}N^m_j +\frac{\del g_{jl}}{\del v^m} N^m_k
 -\frac{\del g_{jk}}{\del v^m} N^m_l \biggr)(v)
\end{equation}
on $TM \setminus \{0\}$,
and the \emph{covariant derivative} of a vector field $Y=\sum_{i=1}^n Y^i (\del/\del x^i)$ is defined as
\[ D_v^w Y :=\sum_{i,j=1}^n
 \Biggl\{ v^j \frac{\del Y^i}{\del x^j}(x)
 +\sum_{k=1}^n \Gamma^i_{jk}(w) v^j Y^k(x) \Biggr\}
 \frac{\del}{\del x^i} \bigg|_x \]
for $v \in T_xM$ with a \emph{reference vector} $w \in T_xM \setminus \{0\}$.
We remark that the functions $\Gamma^i_{jk}$ in \eqref{eq:Gamma} are the coefficients
of the \emph{Chern}(--\emph{Rund}) \emph{connection}.

In the Lorentzian case, $g_{ij}$ is constant in each tangent space (thus, $\Gamma^i_{jk}=\gamma^i_{jk}$)
and the covariant derivative does not depend on the choice of a reference vector.
In the Lorentz--Finsler setting, the following class is worth considering.

\begin{definition}[Berwald spacetimes]\label{df:Ber}
A Finsler spacetime $(M,L)$ is said to be of \emph{Berwald type} (or called a \emph{Berwald spacetime})
if $\Gamma^i_{jk}$ is constant on the slit tangent space $T_xM \setminus \{0\}$ for every $x$.
\end{definition}

By definition, the covariant derivative on a Berwald spacetime is defined independently of the choice of a reference vector.
Thus, in the sequel, reference vectors will be omitted in Berwald spacetimes.
An important property of Berwald spacetimes is that,
for any $C^1$-curve $\eta\colon [0,1] \lra M$ whose velocity does not vanish,
the parallel transport along $\eta$ gives a linear isometry between $(T_{\eta(0)}M,L)$ and $(T_{\eta(1)}M,L)$, i.e., for any vector field $V$ along $\eta$ such that $D_{\dot{\eta}}V \equiv 0$, $L(V)$ is constant (see, e.g., \cite[Proposition~6.5]{Obook} in the positive definite case).
In particular, all tangent spaces are mutually linearly isometric.

\begin{remark}[Metrizability]\label{rm:met}
In the positive definite case,
Szab\'o showed that a Finsler manifold of Berwald type $(M,F)$ admits a Riemannian metric $h$
whose Levi-Civita connection coincides with the Chern connection of $F$,
i.e., the Christoffel symbols of $h$ coincide with $\Gamma^i_{jk}$ of $F$
(see \cite{Sz}, \cite[Exercise~10.1.4]{BCS}).
This is called the (\emph{Riemannian}) \emph{metrizability theorem}.
It is not known whether the metrizability can be generalized to Berwald spacetimes.
In \cite{FHPV}, some counter-examples were constructed for Lorentz--Finsler structures
defined only on a subset of $TM$.
Their discussion is not applicable to Lorentz--Finsler structures defined on the entire tangent bundle as in Definition~\ref{df:LFstr}.
\end{remark}

The \emph{geodesic equation} is $D^{\dot{\eta}}_{\dot{\eta}}\dot{\eta} \equiv 0$.
Then, the \emph{exponential map} is defined in the same way as the Riemannian (or Finsler) case.
For $C^1$-vector fields $V,W$ along a nonconstant geodesic $\eta$, we have
\begin{equation}\label{eq:g_eta}
\frac{\dd}{\dd t}\bigl[ g_{\dot{\eta}}(V,W) \bigr]
 =g_{\dot{\eta}}(D^{\dot{\eta}}_{\dot{\eta}}V,W) +g_{\dot{\eta}}(V,D^{\dot{\eta}}_{\dot{\eta}}W).
\end{equation}
Moreover, for $C^1$-vector fields $V,W$ along a $C^1$-curve $\eta$ such that $V$ is nowhere vanishing,
\begin{equation}\label{eq:g_V}
\frac{\dd}{\dd t}\bigl[ g_V(V,W) \bigr] =g_V(D^V_{\dot{\eta}}V,W) +g_V(V,D^V_{\dot{\eta}}W)
\end{equation}
(see, e.g.,  \cite[(3.1), (3.2)]{LMO1}, \cite[Lemmas~4.8, 4.9]{Obook}).

A $C^{\infty}$-vector field $J$ along a geodesic $\eta$ is called a \emph{Jacobi field}
if it satisfies the \emph{Jacobi equation} $D^{\dot{\eta}}_{\dot{\eta}} D^{\dot{\eta}}_{\dot{\eta}} J +R_{\dot{\eta}}(J) =0$,
where
\[ R_v(w):=\sum_{i,j=1}^n R^i_j(v) w^j \frac{\del}{\del x^i} \bigg|_x \]
for $v,w \in T_xM$ and
\[ R^i_j(v) :=\frac{\del G^i}{\del x^j}(v)
 -\sum_{k=1}^n \biggl\{ \frac{\del N^i_j}{\del x^k}(v) v^k
 -\frac{\del N^i_j}{\del v^k}(v) G^k(v) +N^i_k(v) N^k_j(v) \biggr\} \]
is the \emph{curvature tensor}.
A Jacobi field is also characterized as the variational vector field of a geodesic variation.
One can further expand this as $R^i_j(v)=\sum_{k,l=1}^n R^i_{ljk}(v)v^k v^l$, where
\[ R^i_{ljk}
 :=\Biggl( \frac{\del\Gamma^i_{kl}}{\del x^j} -\sum_{m=1}^n N^m_j \frac{\del\Gamma^i_{kl}}{\del v^m} \Biggr)
 -\Biggl( \frac{\del\Gamma^i_{jl}}{\del x^k} -\sum_{m=1}^n N^m_k \frac{\del\Gamma^i_{jl}}{\del v^m} \Biggr)
 +\sum_{m=1}^n(\Gamma^i_{jm} \Gamma^m_{kl} -\Gamma^i_{km} \Gamma^m_{jl}) \]
(see \cite[(3.3.2)]{BCS}, \cite[(5.10)]{Obook} in the positive definite setting).

For $v \in \Omega_x$ and $w \in T_xM$ linearly independent of $v$,
define the \emph{flag curvature} of the $2$-plane (\emph{flag}) $v \wedge w$ with \emph{flagpole} $v$ as 
\[ \bK(v,w) :=\frac{g_v(R_v(w),w)}{g_v(v,v) g_v(w,w) -g_v(v,w)^2}. \]
Note that the denominator is negative (by the reverse Cauchy--Schwarz inequality for $g_v$ as in \eqref{eq:rCS} below for linearly independent vectors).
We say that the \emph{timelike flag curvature} is nonnegative or $\bK \ge 0$ holds in \emph{timelike directions} if $\bK(v,w) \ge 0$ for all $v$ and $w$ as above.

\begin{remark}[Sign of $\bK$]\label{rm:sign}
The sign of $\bK$ (in timelike directions) above is the same as \cite{BEE,BKOR} and opposite to  \cite{LMO1}.
Thus, the nonnegative flag curvature $\bK \ge 0$ in timelike directions corresponds to
the nonpositive timelike curvature in the sense of \cite{AB,KS,BKR}
(cf.\ \cite[Example~4.9]{KS}, \cite[Remark~2.2]{BKOR}).
The mismatch between nonnegativity and nonpositivity is due to the signature chosen to ensure compatibility with most works in the nonsmooth Lorentzian setting, e.g., \cite{KS,BKR,BKOR}. 
Note also that $\bK \ge 0$ holds in timelike directions for the de Sitter space (Lorentzian pseudo-sphere)
and product spacetimes $(\R \times N, -\dd t^2 +h)$ where $(N,h)$ is an Hadamard manifold. 

\end{remark}

\subsection{Causality conditions and geodesics}

Given $x,y \in M$, we write $x \ll y$ (resp.\ $x<y$)
if there is a timelike (resp.\ causal) curve from $x$ to $y$,
and $x \le y$ means that $x=y$ or $x<y$.
Then, we define the \emph{chronological past} and \emph{future} of $x$ by
\[ I^-(x):=\{y \in M \,|\, y \ll x\}, \qquad I^+(x):=\{y \in M \,|\, x \ll y\}, \]
and the \emph{causal past} and \emph{future} of $x$ by
\[ J^-(x):=\{y \in M \,|\, y \le x\}, \qquad J^+(x):=\{y \in M \,|\, x \le y\}. \]

\begin{definition}[Causality conditions]\label{df:causal}
Let $(M,L)$ be a Finsler spacetime.
\begin{enumerate}[(1)]
\item $(M,L)$ is said to be \emph{chronological} if $x \notin I^+(x)$ for all $x\in M$.
\item We say that $(M,L)$ is \emph{causal} if there is no closed causal curve.
\item $(M,L)$ is said to be \emph{strongly causal} if, for all $x \in M$,
every neighborhood $U$ of $x$ contains another neighborhood $V$ of $x$
such that no causal curve between points in $V$ leaves $V$.
In particular, the topology contains a basis of causally convex sets. 
\item We say that $(M,L)$ is \emph{globally hyperbolic}
if it is strongly causal and, for any $x,y \in M$, $J^+(x) \cap J^-(y)$ is compact.
\end{enumerate}
\end{definition}

Define the \emph{time separation} (also called the \emph{Lorentz--Finsler distance})
$\tau(x,y)$ for $x,y \in M$ by
\[ \tau(x,y) :=\sup_{\eta} \bL(\eta), \qquad
\bL(\eta) :=\int_0^1 F\bigl( \dot{\eta}(t) \bigr) \,\dd t, \]
where the supremum is taken over all (piecewise $C^1$) causal curves $\eta\colon [0,1] \lra M$ from $x$ to $y$. 
We set $\tau(x,y):=0$ if $x \not\le y$.
A curve $\eta$ attaining the above supremum is said to be \emph{maximizing}.
In general, $\tau$ is only lower semi-continuous
and can be infinite (see, e.g., \cite[Proposition~6.7]{Min-Ray}).
In globally hyperbolic Finsler spacetimes, $\tau$ is finite and continuous,
and any pair of points $x,y \in M$ with $x<y$ admits a maximizing geodesic from $x$ to $y$
(Avez--Seifert theorem; see \cite[Propositions~6.8, 6.9]{Min-Ray}).

\section{Concavity of Berwald spacetimes}\label{sc:Bconc}

Let $(M,L)$ be a Finsler spacetime.
For each $x \in M$, there exists a \emph{convex normal neighborhood} $U \subset M$ of $x$ in the sense that, for every $y \in U$, the exponential map $\exp_y$ gives a $C^1$-diffeomorphism from an open star-shaped set $\mathcal{U}_y \subset T_yM$ onto $U$ (see \cite[Theorem~4]{Min-convex}).
In particular, for any $y,z \in U$, there is a unique geodesic from $y$ to $z$ contained in $U$.
In this and the following sections, since we shall deal only with local structures (governed by curvature), we are concerned with causal relations and the time separation restricted to $U$.
They coincide with the global ones if $(M,L)$ is strongly causal (by taking smaller $U$ if necessary).

The following is a Lorentzian analog to the \emph{convexity} of metric spaces
(also called the \emph{Busemann nonpositive curvature}; recall the introduction), inspired by \cite{BKR,EG}.

\begin{definition}[Concavity]\label{df:conc}
Let $(M,L)$ be a Finsler spacetime.
\begin{enumerate}[(1)]
\item\label{it:ccv1}
We say that $(M,L)$ is \emph{locally timelike concave} if every $x \in M$ admits a convex normal neighborhood $U$
such that, for any timelike geodesics $\eta,\xi\colon [0,1] \lra U$ with $\eta(t) \ll_U \xi(t)$ or $\eta(t)=\xi(t)$ at $t=0,1$, we have
\[ \tau_U \bigl( \eta(t),\xi(t) \bigr) \ge (1-t)\tau_U \bigl( \eta(0),\xi(0) \bigr) +t\tau_U \bigl( \eta(1),\xi(1) \bigr) \]
for all $t \in [0,1]$.

\item\label{it:ccv2}
We say that $(M,L)$ is \emph{locally concave} if every $x \in M$ admits a convex normal neighborhood $U$
such that, for any geodesics $\eta,\xi\colon [0,1] \lra U$ with $\eta(t) \ll_U \xi(t)$ or $\eta(t) = \xi(t)$ at $t=0,1$, we have
\[ \tau_U \bigl( \eta(t),\xi(t) \bigr) \ge (1-t)\tau_U \bigl( \eta(0),\xi(0) \bigr) +t\tau_U \bigl( \eta(1),\xi(1) \bigr) \]
for all $t \in [0,1]$.
\end{enumerate}
Here we respectively denote by $\ll_U$ and $\tau_U$ the timelike relation and time separation on $(U,L|_U)$.
\end{definition}

Note that we allow mixed cases in Definition~\ref{df:conc}, i.e., $\eta(0) \ll_U \xi(0)$ and $\eta(1)=\xi(1)$ (or vice versa). 

\begin{remark}[Causal character of geodesics]
\label{rm:conc}
Note that $\eta,\xi$ in \eqref{it:ccv2} above are not necessarily causal, and that local concavity clearly implies local timelike concavity.
It was proved in \cite[Proposition~6.1]{BKR} that a Lorentzian pre-length space of timelike nonpositive curvature in the sense of strict causal triangle comparison is locally timelike concave (recall Remark~\ref{rm:sign} for the sign convention in timelike curvature bounds).
In such a synthetic setting, it is natural to consider only timelike geodesics, since we do not know how to define spacelike geodesics.
Then, in \cite[Definition~2.1]{EG}, local timelike concavity as in \eqref{it:ccv1} above is termed local concavity.
For Berwald spacetimes, we will see that local timelike concavity and local concavity are, in fact, equivalent.
\end{remark}

To show the local concavity from $\bK \ge 0$,
we can apply a similar calculation to the positive definite case (cf.\ \cite{KVK}, \cite[Theorem~8.30]{Obook}),
while causality needs to be addressed.
Concerning the next proposition, a related result in the Lorentzian case can be found in \cite[Proposition~11.15(1)]{BEE}
(we will need an additional discussion relying on the Berwald condition to deal with the difference between $\bK(T,V)$ and $\bK(V,T)$).

\begin{proposition}[Concave norm]
\label{prop: concave length}
Let $(M,L)$ be a Berwald spacetime
of nonnegative flag curvature $\bK \ge 0$ in timelike directions. 
If $\sigma \colon [0,1] \times (-\ve,\ve) \lra U$ is a variation in a convex normal neighborhood $U$ such that $\sigma_s(t):=\sigma(t,s)$ is a geodesic for each $s \in (-\ve,\ve)$ and that $V(t,s):=\partial_s\sigma(t,s)$ is timelike for all $(t,s)$, then, for every $s$, the function $t\longmapsto F(V(t,s))$ is concave. 
\end{proposition}

\begin{proof}
Put $T(t,s):=\del_t \sigma(t,s)$.
Observe from \eqref{eq:LtoF} and \eqref{eq:g_V} that
\[ \frac{\del^2 [F^2(V)]}{\del t^2} =-\frac{\del^2 [g_V(V,V)]}{\del t^2}
 =-2g_V (V,D_T D_T V) -2g_V(D_T V,D_T V) \]
(recall that we can omit reference vectors in Berwald spacetimes).
Since $V(\cdot,s)$ is a Jacobi field by the hypothesis for $\sigma$, we find
\begin{equation}\label{eq:secondVariationDD}
g_V(V,D_T D_T V) =-g_V \bigl( V,R_T(V) \bigr)
 =-g_V \Biggl( V,\sum_{i,j,k,l=1}^n R^i_{ljk}(T) V^j T^k T^l \frac{\del}{\del x^i} \Biggr). 
\end{equation}
Note that $R^i_{ljk}(T)=R^i_{ljk}(V)$ by the Berwald condition, and that
\begin{equation}\label{eq:secondVariationK}
\begin{split}
g_V\Biggl( V,\sum_{i,j,k,l=1}^n R^i_{ljk}(V) V^j T^k T^l \frac{\del}{\del x^i} \Biggr)
&= g_V\bigl( R_V(T),T \bigr) \\
&= \bK(V,T) \bigl\{ g_V(V,V) g_V(T,T) -g_V(V,T)^2 \bigr\},
\end{split}
\end{equation}
where the first equation follows from \cite[Lemma~5.15(iv)]{Obook} (which is valid regardless of the signature of metrics).

Now, $\bK(V,T) \ge 0$ by hypothesis and, 
choosing a local chart $(x^i)_{i=1}^n$ which is $g_V$-orthonormal with $\del/\del x^1=V/F(V)$, we observe
\begin{equation}\label{eq:rCS}
F^2(V) g_V(T,T) +g_V(V,T)^2
 =V_1^2 \Biggl( -T_1^2 +\sum_{i=2}^n T_i^2 \Biggr) +(V_1 T_1)^2 \ge 0.
\end{equation}
Hence, we obtain
\begin{equation}
\begin{split}\label{eq:secondVariationFirstStep}
\frac{\del^2 [F(V)]}{\del t^2}
& = \frac{1}{2F(V)} \frac{\del^2 [F^2(V)]}{\del t^2}
 -\frac{1}{F(V)} \biggl( \frac{\del [F(V)]}{\del t} \biggr)^2 \\
& \le  -\frac{g_V(D_T V,D_T V)}{F(V)}
 -\frac{1}{F(V)} \biggl( \frac{g_V(V,D_T V)}{F(V)} \biggr)^2 \\
& = - \frac{1}{F^3(V)}
 \Bigl\{ F^2(V) g_V(D_T V,D_T V) +g_V(V,D_T V)^2 \Bigr\}.
\end{split}
\end{equation}
Since $F^2(V) g_V(D_T V,D_T V) +g_V(V,D_T V)^2 \ge 0$ by \eqref{eq:rCS}, this implies that $\partial^2[F(V)]/\partial t^2 \le 0$.
Therefore, $F(V(t,s))$ is concave in $t$.
\end{proof}

The following two lemmas are concerned with the causality issue.

\begin{lemma}[Timelike homotopy]
\label{lem: timelike homotopy}
Let $U$ be a convex normal neighborhood in a Finsler spacetime $(M,L)$.
Then, for any timelike curves $\alpha,\beta \colon [0,1] \lra U$,
there exists a homotopy $h\colon [0,1] \times [0,1] \lra U$ such that $h(0;s)=\alpha(s)$, $h(1;s)=\beta(s)$ and $s \longmapsto h(t;s)$ is timelike for each $t$,
and that $h$ and $\partial_s h$ are continuous.
\end{lemma}

\begin{proof}
Put $x=\alpha(0)$, $y=\beta(0)$ and let $\eta\colon [0,1] \lra U$ be the geodesic from $x$ to $y$.
Consider a smooth timelike vector field $X$ that exists by time orientation and,
replacing it with $cX$ for small $c>0$ if necessary, we assume that $\exp_{\eta(t)}(sX(\eta(t))) \in U$ for all $t,s\in[0,1]$.

We first choose a homotopy $h_1$ from $\alpha$ to the geodesic from $x$ to $\alpha(\ve)$ for small $\ve>0$, and take $v \in \Omega_x$ with $\alpha(\ve)=\exp_x(v)$.
Then, set
\[ h_2(t;s):=\exp_x \Bigl( s\bigl( (1-t)v+tX(x) \bigr) \Bigr),\quad
h_3(t;s):=\exp_{\eta(t)}\Bigl( sX\bigr( \eta(t) \bigr) \Bigr) \]
for $t,s \in [0,1]$.
Concatenating $h_1$, $h_2$ and $h_3$,
and connecting $h_3(1;s)=\exp_y(sX(y))$ and $\beta$ in the same way, we obtain a homotopy that satisfies the desired property.
\end{proof}

\begin{lemma}[Timelike variation]
\label{lem: VtimeGen}
Let $(M,L)$ be a Berwald spacetime of nonnegative flag curvature $\bK \ge 0$ in timelike directions,
and $U \subset M$ be a convex normal neighborhood.
Given timelike curves $\alpha,\beta\colon [0,1] \lra U$, let $t\longmapsto \sigma(t,s)$ be the geodesic from $\alpha(s)$ to $\beta(s)$ in $U$.
Then, $\partial_s \sigma$ is timelike and $t \longmapsto F(\del_s \sigma(t,s))$ is concave for each $s \in [0,1]$.
\end{lemma}

\begin{proof}
Let $h(\lambda;s)$ be a homotopy between $\alpha$ and $\beta$ given by Lemma~\ref{lem: timelike homotopy}, i.e.,
$h(0;s)=\alpha(s)$, $h(1;s)=\beta(s)$ and $\partial_s h$ is always timelike.
For each $\lambda \in [0,1]$, we consider the variation $\Xi(\lambda;\cdot,\cdot) \colon [0,1] \times [0,1] \lra U$ such that $t \longmapsto \Xi(\lambda;t,s)$ is the geodesic from $\Xi(\lambda;0,s):=h(\lambda;s)$ to $\Xi(\lambda;1,s):=\beta(s)$ in $U$.
Observe that $\Xi(0;t,s)=\sigma(t,s)$ and $\Xi(1;t,s)= \beta(s)$.

Now, we set
\[ \lambda_0:=\sup\bigl\{ \lambda \in [0,1] \mid \del_s \Xi(\lambda;t,s) \text{ is not timelike for some } t,s \in [0,1] \bigr\}. \]
Note that $\lambda_0 < 1$, since $\partial_s\Xi(1;t,s)=\dot\beta(s)$ is timelike and being timelike is an open condition.
%
Then, for any $\lambda \in (\lambda_0,1]$, since $\del_s \Xi(\lambda;t,s)$ is timelike for all $t,s \in [0,1]$, it follows from Proposition~\ref{prop: concave length} that
$F(\del_s \Xi(\lambda;t,s))$ is concave in $t$.
Taking the limit as $\lambda \to \lambda_0$, we find that $F(\del_s \Xi(\lambda_0;t,s))$ is concave in $t$.
Hence,
\begin{align*}
F\bigl( \del_s \Xi(\lambda_0;t,s) \bigr)
&\ge (1-t)F\bigl( \del_s \Xi(\lambda_0;0,s) \bigr) +tF\bigl( \del_s \Xi(\lambda_0;1,s)\bigr) \\
&= (1-t)F\bigl( \del_s h(\lambda_0;s)\bigr) +tF\bigl( \dot{\beta}(s) \bigr) >0
\end{align*}
for all $t,s \in [0,1]$.
This shows that $\partial_s\Xi(\lambda_0;t,s)$ is also timelike for all $t,s \in [0,1]$.

If $\lambda_0>0$, then we deduce from the openness of being a timelike vector that $\partial_s\Xi(\lambda;t,s)$ is timelike for all $t,s \in [0,1]$ and $\lambda \in (\lambda_0-\ve,1]$ for sufficiently small $\ve>0$, contradicting the definition of $\lambda_0$.
Consequently, $\partial_s\Xi(\lambda;t,s)$ is timelike for all $\lambda,t,s \in [0,1]$.
Taking $\lambda=0$ implies that $\del_s \sigma$ is timelike, and $F(\del_s \sigma(t,s))=F(\del_s \Xi(0;t,s))$ is concave in $t$. 
\end{proof}

We are ready to prove \eqref{it:K>0} $\Rightarrow$ \eqref{it:ccv} in Theorem~\ref{th:main}.

\begin{theorem}[$\bK \ge 0$ implies concavity]
\label{th:Bconc}
Let $(M,L)$ be a Berwald spacetime of nonnegative flag curvature $\bK \ge 0$ in timelike directions.
Then, $(M,L)$ is locally concave.
\end{theorem}

\begin{proof}
Let $U \subset M$ be a convex normal neighborhood.
We first consider geodesics $\eta,\xi\colon [0,1] \lra U$ with $\eta(0) \ll_U \xi(0)$ and $\eta(1) \ll_U \xi(1)$.
Then, take the geodesics $\alpha$ from $\eta(0)$ to $\xi(0)$ and $\beta$ from $\eta(1)$ to $\xi(1)$,
and let $\sigma\colon [0,1] \times [0,1] \lra U$ be the variation such that $\sigma(0,s)=\alpha(s)$, $\sigma(1,s)=\beta(s)$, and that $t \longmapsto \sigma(t,s)$ is the geodesic from $\alpha(s)$ to $\beta(s)$.
Then, it follows from Lemma~\ref{lem: VtimeGen} that $\del_s \sigma$ is timelike and $F(\del_s \sigma(t,s))$ is concave in $t$.
This implies that
\[ \bL(t):=\int_0^1 F\bigl( \del_s \sigma(t,s) \bigr) \,\dd s \] is concave, and hence
\[ \tau_U \bigl( \eta(t),\xi(t) \bigr) \geq \bL(t)
\geq (1-t)\bL(0) + t\bL(1)
=(1-t)\tau_U \bigl( \eta(0),\xi(0) \bigr) + t\tau_U \bigl( \eta(1),\xi(1) \bigr), \]
as desired.

The case of $\eta(0)=\xi(0)$ or $\eta(1)=\xi(1)$ is obtained as the limit, thanks to the continuity of geodesics in their endpoints (in $U$). 
\end{proof}

The converse implication (from concavity to the Berwald condition and $\bK \ge 0$)
was established in the positive definite case in \cite{IL}; however,
their argument cannot apply to the Lorentzian setting (see Subsection~\ref{ssc:nonB} for more details).
Here we consider a weaker result in the spirit of \cite{KK}, putting the Berwald condition in the assumption. 
To this end, we need another preliminary lemma. 

\begin{lemma}[Concavity yields concave norm]
\label{lem:ConcavityJacobiField}
Let $J$ be a timelike Jacobi field along a timelike geodesic $\eta\colon [0,\delta] \lra U$ in a convex normal neighborhood $U \subset M$.
If $\tau$ is timelike concave in $U$ $($in the sense of Definition~$\ref{df:conc})$, then $t \longmapsto F(J(t))$ is concave.
Similarly, if $\tau$ is concave in $U$, then $F(J)$ is concave for every timelike Jacobi field $J$ along any geodesic $\eta$ in $U$.
\end{lemma}

\begin{proof}
Let $\sigma\colon [0,\delta] \times (-\ve,\ve) \lra U$ be any geodesic variation associated with $J$, i.e.,
$\sigma_s:=\sigma(\cdot,s)$ is a geodesic for every $s$, $\sigma_0=\eta$,
and $\del_s \sigma(t,0)=J(t)$ for all $t$.
Note that $\sigma_s$ is timelike for $s$ close to $0$.
Moreover, since $J(t)$ is timelike,
$t \longmapsto \tau_U(\eta(t),\sigma_s(t))$ is concave for small $s>0$ by timelike concavity.
Then, the concavity of $t \longmapsto F(J(t))$ follows from
\[ \lim_{s \downarrow 0} \frac{\tau_U(\eta(t),\sigma_s(t))}{s} =F\bigl( J(t) \bigr). \]
The latter assertion is shown in the same way.
\end{proof}

The next theorem shows \eqref{it:c-ccv} $\Rightarrow$ \eqref{it:K>0} in Theorem~\ref{th:main}.

\begin{theorem}[Timelike concavity implies $\bK \ge 0$]\label{th:BNPC}
Let $(M,L)$ be a Berwald spacetime.
If $(M,L)$ is locally timelike concave, then we have $\bK \ge 0$ in timelike directions.
\end{theorem}

\begin{proof}
Suppose to the contrary that $g_v(R_v(w),w) > 0$ holds for some linearly independent timelike vectors $v, w \in \Omega_x$.
Let $\eta\colon [0,\delta] \lra M$ be the geodesic with $\dot{\eta}(0) = w$
and $J$ be the Jacobi field along $\eta$ with $J(0) = v$ and $D_{\dot{\eta}} J(0) =0$.
Then, we have
\[ D_{\dot{\eta}} D_{\dot{\eta}} J(0) =-R_w(v) \]
and, by \eqref{eq:g_V},
\[ \frac{\dd}{\dd t} \Bigl[ L\bigl( J(t) \bigr) \Bigr]_{t=0}
 = \frac{1}{2} \frac{\dd}{\dd t} \Bigl[ g_J\bigl( J(t),J(t) \bigr) \Bigr]_{t=0}
 = g_v \bigl( v,D_{\dot{\eta}} J(0) \bigr) =0. \]
Combining these and recalling $F=\sqrt{-2L}$, we obtain
\begin{align*}
\frac{\dd^2}{\dd t^2}\Bigl[ F\bigl( J(t) \bigr) \Bigr]_{t=0}
&= -\frac{\dd}{\dd t} \Biggl[ \frac{1}{\sqrt{-2L(J(t))}} \cdot \frac{\dd}{\dd t} \Bigl[ L\bigl( J(t) \bigr) \Bigr] \Biggr]_{t=0} \\
&= -\frac{1}{F(v)} \cdot \frac{\dd}{\dd t} \Bigl[ g_J \bigl( J(t),D_{\dot{\eta}} J(t) \bigr) \Bigr]_{t=0} \\
&= \frac{1}{F(v)} \cdot g_v \bigl( v,R_w(v) \bigr).
\end{align*}
However, as we saw in \eqref{eq:secondVariationDD} and \eqref{eq:secondVariationK}, the Berwald condition ensures that
\[ g_v \bigl( v,R_w(v) \bigr) =g_v \bigl( R_v(w),w \bigr) >0. \]
This contradicts Lemma~\ref{lem:ConcavityJacobiField}, and completes the proof.
\end{proof}

\begin{remark}[Globalization]
\label{rm:global}
\begin{enumerate}[(a)]
\item
In the case of metric spaces, Busemann's convexity is known to enjoy the \emph{globalization} property (called the Cartan--Hadamard theorem):
If a complete, connected metric space $(X,d)$ is locally convex, then its universal cover $\widetilde{X}$ with its induced length metric is globally convex
(see \cite[Theorem~II.4.1(1)]{BH} and the references therein).
In particular, if $(X,d)$ is itself a simply connected length space, then it is globally convex.

\item
In the Lorentzian setting, an analogous statement was proved in \cite[Theorem~4.9]{EG}:
If a globally hyperbolic, regular Lorentzian length space $X$ is locally concave and future one-connected, then it is globally concave. Here, global hyperbolicity replaces the metric completeness assumption, and future one-connectedness refers to the requirement that $X$ be timelike path-connected and any pair of future-directed timelike curves with common endpoints be homotopic through a family of timelike curves.
Recall from Remark~\ref{rm:conc} that concavity in \cite{EG} means timelike concavity in the sense of Definition~\ref{df:conc}.
\end{enumerate}
\end{remark}

\section{Convex capsules}\label{sc:capsule}

We next consider another kind of characterization of the concavity condition for Berwald spacetimes.
Inspired by \cite{KK}, we introduce the following.

\begin{definition}[Convex capsules]\label{df:capsule}
A Finsler spacetime $(M,L)$ is said to have \emph{convex future capsules} if, for every $x \in M$, there exists a convex normal neighborhood $U_x$ of $x$ and $r>0$ such that the set
\[ K_{\geq r}^+(\gamma) :=\bigcup_{y \in \gamma} H_{\geq r}^+(y) \]
is geodesically convex for any geodesic $\gamma$ contained in $U_x$, where
\[ H_{\geq r}^+(y):=\{z \in U_x \mid \tau_{U_x}(y,z) \geq r\}. \]
To have \emph{convex past capsules} is defined in the same way with
\[ K_{\geq r}^-(\gamma) :=\bigcup_{y \in \gamma} H_{\geq r}^-(y), \qquad
H_{\geq r}^-(y):=\{z \in U_x \mid \tau_{U_x}(z,y) \geq r\}. \]
\end{definition}

In \cite[Definition~3]{KK}, they considered the convexity of a neighborhood of a short geodesic in a metric space, whose shape indeed resembles a capsule.
In our definition, $K_{\ge r}^\pm (\gamma)$ does not look like a capsule (see Figure~\ref{fig: convex past capsule}), but we keep the name and notation to make the correspondence with \cite{KK} transparent. 

\begin{figure}
\begin{center}
\begin{tikzpicture}
\draw (0,0) -- (3,0);
\draw (0,-1) -- (3,-1);

\draw[dashed,scale=1,domain=-180:180,smooth,variable=\t] plot ({4*cos(\t)+1.5},{2*sin(\t)-1});

\draw[scale=1,domain=-70:0,smooth,variable=\t] plot ({tan(\t)},{-sec(\t)});

\draw[scale=1,domain=0:70,smooth,variable=\t] plot ({tan(\t)+3},{-sec(\t)});

\begin{scriptsize}

\coordinate [circle, fill=black, inner sep=0.5pt] (bx1) at (0,0);
\coordinate [circle, fill=black, inner sep=0.5pt] (bx1) at (3,0);

\coordinate [label=90: {$\gamma$}] (curve) at (1.5,0);

\coordinate [label=0: {$U_x$}] (nbhd) at (4.5,-1);

\coordinate [label=90: {$K_{\geq r}^-(\gamma)$}] (capsule) at (1.5,-2.5);
\end{scriptsize}

\end{tikzpicture}
\end{center}
\caption{The convex past capsule associated to a geodesic $\gamma$.}
\label{fig: convex past capsule}
\end{figure}
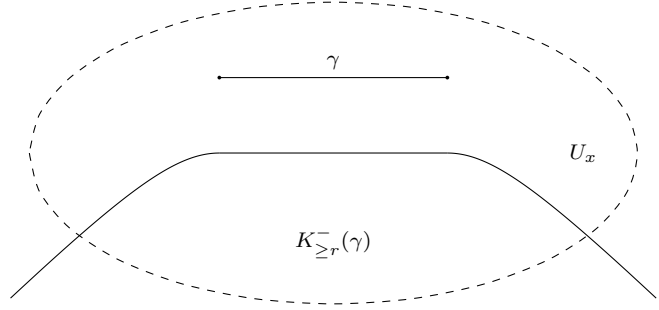

\begin{proposition}[Concavity implies convex capsules]
\label{pr:ccv-cap}
If a Berwald spacetime $(M,L)$ is locally concave, then it has convex future and past capsules.
\end{proposition}

\begin{proof}
Given $x \in M$, let $U$ be a convex normal neighborhood of $x$ and $\gamma$ be any geodesic in $U$.
Take $z_1,z_2 \in K:=K_{\geq r}^+(\gamma) \subset U$ and let $\alpha\colon [0,1] \lra U$ be the geodesic from $z_1$ to $z_2$.
Since $z_1,z_2 \in K$, we find $t_1,t_2 \in [0,1]$ such that $\tau_U(\gamma(t_1),z_1) \geq r$ and $\tau_U(\gamma(t_2),z_2) \geq r$.
If $t_1 \le t_2$, then, we infer from the assumed concavity that, for any $s \in [0,1]$,
\[ \tau_U\Bigl( \gamma\bigl( (1-s)t_1 +st_2 \bigr),\alpha(s) \Bigr)
\geq (1-s)\tau_U\bigl( \gamma(t_1),z_1 \bigr) + s\tau_U\bigl( \gamma(t_2),z_2 \bigr)
\geq r. \]
This shows that $\alpha(s) \in K$.
If $t_2<t_1$, then we observe from the Berwald condition that the reversal of $\gamma|_{[t_2,t_1]}$ is still a geodesic 
from $\gamma(t_1)$ to $\gamma(t_2)$ and the above argument shows $\alpha(s)\in K$.
This completes the proof for the future case.

We can see that concavity implies convex past capsules in the same way.
\end{proof}

In the converse implication, we can roughly follow the argument in \cite{KK} ((e) $\Rightarrow$ (a) of Theorem~1). 

\begin{proposition}[Convex capsules imply $\bK \ge 0$]
\label{pr:cap-nnc}
Let $(M,L)$ be a Berwald spacetime having convex future capsules.
Then, we have $\bK \ge 0$ in timelike directions.
Similarly, if $(M,L)$ is a Berwald spacetime having convex past capsules, then $\bK \ge 0$ in timelike directions.
\end{proposition}

\begin{proof}
Let $x \in M$.
It is sufficient to show $\bK(v,w) \ge 0$ for any $v \in \Omega_x$ and $w \in T_xM$ satisfying $F(v)=1$, $g_v(v,w)=0$ and $g_v(w,w)=1$.
For small $r>0$, let $\eta\colon [0,r] \lra U_x$ be the geodesic with $\dot{\eta}(0)=v$ and $W$ be the vector field along $\eta$ such that $W(0)=w$ and $D_{\dot{\eta}}W \equiv 0$.
(Recall that we can drop reference vectors in Berwald spacetimes.)
It follows from \eqref{eq:g_eta} or \eqref{eq:g_V} that
\[
\frac{\dd}{\dd t} \Bigl[ g_{\dot{\eta}}(\dot{\eta},W) \Bigr] =g_{\dot{\eta}}(D_{\dot{\eta}} \dot{\eta},W) +g_{\dot{\eta}}(\dot{\eta},D_{\dot{\eta}} W)=0.
\]
Hence, $g_{\dot{\eta}}(\dot{\eta}(t),W(t))=0$ holds for all $t \in [0,r]$.

Consider the variation $\sigma\colon [0,r] \times (-\ve,\ve) \lra U_x$ such that $\sigma_t(s):=\sigma(t,s)$ is the geodesic with $\dot{\sigma}_t(0)=W(t)$, and define $T:=\del_t \sigma$ and $V:=\del_s \sigma$ (thus $V(t,0)=W(t)$).
Note that $\sigma_r$ is tangent to $H_{\ge r}^+(x)$ at $\eta(r)$ by $g_{\dot{\eta}}(\dot{\eta}(r),W(r))=0$ (see Figure~\ref{fig: geodesic variation sigma}).
Moreover, we deduce from $g_{\dot{\eta}}(\dot{\eta},W) \equiv 0$ and the first variation formula for $\bL(s):=\int_0^r F(T(t,s)) \,\dd t$ (see \cite[Lemma~2.6]{BO}) that $\bL'(0)=0$
and $\eta(t) \not\in K$ for $t \in [0,r)$, where
\[ K:=K_{\ge r}^+(\sigma_0) 
=\bigcup_{s \in (-\ve,\ve)}H_{\ge r}^+\bigl( \sigma_0(s) \bigr). \]
Since $K$ is a geodesically convex set (by hypothesis) including $H_{\ge r}^+(x)$, we infer that $\sigma_r$ is also tangent to $K$.
This implies that, by the convexity of $K$, $\sigma_r$ does not enter the interior of $K$.
In particular, $\sigma_r(s)$ is not in the interior of $H_{\ge r}^+(\sigma_0(s))$.
Hence, 
\[ \bL(s) \le \tau_{U_x}\bigl( \sigma_0(s),\sigma_r(s) \bigr) \le r. \]
Combining this with $\bL(0)=r$ and $\bL'(0)=0$, we obtain $\bL''(0) \le 0$.

\begin{figure}
\begin{center}
\begin{tikzpicture}
\draw (0,0) .. controls (1,.3) and (2,-.3) .. (3,0);
\draw (0,3) .. controls (1,3.3) and (2,2.7) .. (3,3);

\draw (0,0) .. controls (.2,1.5) and (-.2,2.5) .. (0,3);
\draw (3,0) .. controls (3.2,1.5) and (2.8,2.5) .. (3,3);

\draw (1.5,0) .. controls (1.7,1.5) and (1.3,2.5) .. (1.5,3);
\draw[->] (1.5,0) -- (1.7,1.5);
\draw[->] (1.5,0) -- (3,-0.2);

\draw[scale=1,domain=-70:70,smooth,variable=\t] plot ({tan(\t)+1.6},{sec(\t)+1.99});

\draw[dotted] (0.1,1) .. controls (1.1,1.3) and (2.1,0.7) .. (3,1.1);
\draw[dotted] (0.1,2) .. controls (1.1,2.3) and (2.1,1.7) .. (3,2.1);

\draw[dotted] (1,0.1) .. controls (1.2,1.6) and (0.8,2.6) .. (1,3.1);
\draw[dotted] (2,-0.05) .. controls (2.2,1.45) and (1.8,2.45) .. (2,2.95);
\begin{scriptsize}
\coordinate [circle, fill=black, inner sep=0.5pt, label=270: {$\sigma(0,0)=\eta(0)=x$}] (bx1) at (1.5,0);
\coordinate [circle, fill=black, inner sep=0.5pt, label=45: {$\sigma(0,\varepsilon)$}] (bx1) at (3,0);
\coordinate [circle, fill=black, inner sep=0.5pt, label=45: {$\sigma(r,\varepsilon)$}] (bx1) at (3,3);
\coordinate [circle, fill=black, inner sep=0.5pt, label=180: {$\sigma(0,-\varepsilon)$}] (bx1) at (0,0);
\coordinate [circle, fill=black, inner sep=0.5pt, label=135: {$\sigma(r,-\varepsilon)$}] (bx1) at (0,3);

\coordinate [circle, fill=black, inner sep=0.5pt] (bx1) at (1.5,3);
\coordinate [label=180: {$\eta$}] (curve) at (1.5,1.8);
\coordinate [label=90: {$H^+_{\geq r}(x)$}] (curve) at (1.5,3.5);
\coordinate [label=315: {$\dot \sigma_0(0)=W(0)=w$}] (V) at (3,-0.2);
\coordinate [label=45: {$\dot \eta(0)=v$}] (V) at (1.7,1.5);
\end{scriptsize}

\end{tikzpicture}
\end{center}
\caption{The geodesic variation $\sigma$ and a tangent hyperbola.}
\label{fig: geodesic variation sigma}
\end{figure}
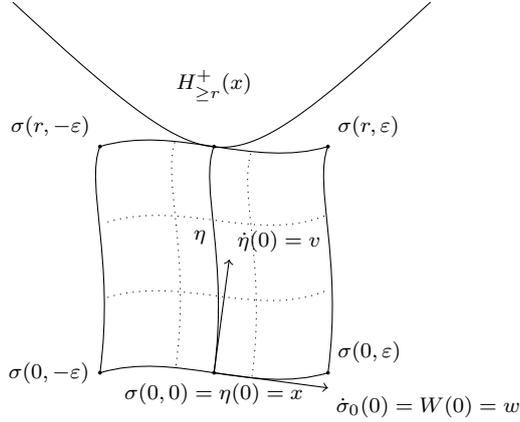

For each $t \in [0,r]$, $s \longmapsto T(t,s)$ is a Jacobi field along $\sigma_t$ by construction.
Thus, we find that
\[ \frac{\del^2 [F^2(T)]}{\del s^2}\Big|_{s=0}
= 2g_{\dot{\eta}} \bigl( \dot{\eta},R_W(\dot{\eta}) \bigr) -2g_{\dot{\eta}}(D_W T,D_W T) \]
(recall $V(t,0)=W(t)$).
In the first term, by \eqref{eq:secondVariationDD}, \eqref{eq:secondVariationK} and $D_{\dot{\eta}}W \equiv 0$, we infer that
\[ g_{\dot{\eta}} \bigl( \dot{\eta},R_W(\dot{\eta}) \bigr)
=\bK(\dot{\eta},W) \bigl\{ g_{\dot{\eta}}(\dot{\eta},\dot{\eta}) g_{\dot{\eta}}(W,W) -g_{\dot{\eta}}(\dot{\eta},W)^2 \bigr\}
=-\bK(\dot{\eta},W). \]
As to the second term, we have
\[ D_V T(t,0) =D_T V(t,0) =D_{\dot{\eta}} W(t) =0. \]
Hence, by a similar calculation to \eqref{eq:secondVariationFirstStep} (using $F(\dot{\eta}) \equiv 1$), we obtain
\begin{align*}
\bL''(0) &= \int_0^r \biggl\{ \frac{1}{2} \frac{\del^2 [F^2(T)]}{\del s^2}\Big|_{s=0}
-\biggl( \frac{\del [F(T)]}{\del s} \Big|_{s=0} \biggr)^2 \biggr\} \,\dd t \\
&=-\int_0^r \bigl\{  \bK(\dot{\eta},W)
+g_{\dot{\eta}}(\dot{\eta},D_W T)^2  \bigr\} \,\dd t \\
&=-\int_0^r \bK(\dot{\eta},W) \,\dd t.
\end{align*}
Since $\bK$ is continuous, letting $r \to 0$ yields $\bK(v,w) \geq 0$.

The case of convex past capsules is shown in the same way, using the geodesic $\eta\colon [-r,0] \lra U_x$ with $\dot{\eta}(0)=v$.
Alternatively, we can reduce the past case for $L$ to the future case with respect to the reverse Lorentz--Finsler structure $\overline{L}(v):=L(-v)$, for the curvature bound $\bK \ge 0$ (in timelike directions) for $L$ is equivalent to that for $\overline{L}$
(cf.\ \cite[Remark~8.14]{LMO1}, \cite[\S 2.5]{Obook}).
\end{proof}

\begin{remark}[Spacelike geodesics]
\label{rm:spacelike}
Note that, by construction, $\dot{\sigma}_t(0)=W(t)$ is spacelike with respect to $g_{\dot{\eta}(t)}$.
This is the reason why we introduced not only local timelike concavity but also local concavity in Definition~\ref{df:conc}, although it is desirable to deal only with causal curves.
In fact, for spacelike geodesics, nothing can be said about their maximality or minimality;
thus, it is not known how to consider spacelike geodesics in the synthetic setting like \cite{KS,BKR,EG}.
\end{remark}

The proof of Theorem~\ref{th:main} is now complete:
\eqref{it:K>0} $\Rightarrow$ \eqref{it:ccv} by Theorem~\ref{th:Bconc},
\eqref{it:ccv} $\Rightarrow$ \eqref{it:c-ccv} is obvious,
\eqref{it:c-ccv} $\Rightarrow$ \eqref{it:K>0} by Theorem~\ref{th:BNPC},
\eqref{it:ccv} $\Rightarrow$ \eqref{it:f-cap} and \eqref{it:p-cap} by Proposition~\ref{pr:ccv-cap},
\eqref{it:f-cap} or \eqref{it:p-cap} $\Rightarrow$ \eqref{it:K>0} by Proposition~\ref{pr:cap-nnc}.

\section{Further problems}\label{sc:problem}

This section is devoted to discussions on some remaining open problems.

\subsection{From concavity to the Berwald condition}\label{ssc:nonB}

In this subsection, we consider general (non-Berwald) Finsler spacetimes and explain difficulties in following the argument in \cite{IL} in the positive definite case.
Notice that Lemma~3.1 and Proposition~4.1 in \cite{IL} can be generalized as follows.

\begin{lemma}[{cf.\ \cite[Lemma 3.1]{IL}}]
\label{lm:IL3.1}
Let $(M,L)$ be a locally concave Finsler spacetime.
Then, for every parallel vector field $V$ along any nonconstant geodesic $\eta\colon [0,1] \lra M$ with $V(0) \in \Omega_{\eta(0)}$ $($i.e., $D^{\dot{\eta}}_{\dot{\eta}}V \equiv 0)$, $L(V)$ is constant.
In particular, $V(t) \in \Omega_{\eta(t)}$ for all $t \in [0,1]$.
\end{lemma}

\begin{proof}
Let $J$ be the Jacobi field along $\eta$ with $J(0)=0$ and $D_{\dot{\eta}}^{\dot{\eta}} J(0)=V(0)$.
Observe that
\[ D_{\dot{\eta}}^{\dot{\eta}} D_{\dot{\eta}}^{\dot{\eta}} J(0) =-R_{\dot{\eta}}(0) =0. \]
Comparing this with $D_{\dot{\eta}}^{\dot{\eta}}[tV(t)] = V(t)$ and $D_{\dot{\eta}}^{\dot{\eta}} D_{\dot{\eta}}^{\dot{\eta}} [tV(t)]=0$, we find $J(t) =tV(t) +O(t^3)$.
It follows that
\[ F\bigl( J(t) \bigr) =tF\bigl( V(t) \bigr) +O(t^3) \quad \text{as}\ t \to 0. \]
Combining this with the concavity of $t \longmapsto F(J(t))$ by Lemma~\ref{lem:ConcavityJacobiField}, we obtain $F(J(t)) \le tF(V(0))$.
Hence, we have
\[ F\bigl( V(t) \bigr) =\frac{1}{t} F\bigl( J(t) \bigr) +O(t^2) \le F\bigl( V(0) \bigr) +O(t^2). \]
This implies
\[ \frac{\dd}{\dd t} \Bigl[ F\bigl( V(t) \bigr) \Bigr]_{t=0} \le 0, \]
and similarly $\frac{\dd}{\dd t}[F(V(t))] \le 0$ for all $t \in (0,1)$ such that $V(t)$ is timelike. 

The same argument for the reverse Lorentz--Finsler structure $\overline{L}(v):=L(-v)$
shows that $\frac{\dd}{\dd t}[F(V(t))] \ge 0$ for all $t \in (0,1]$ such that $V(t)$ is timelike.
Indeed, $\bar{\eta}(t):=\eta(1-t)$ is a geodesic for $\overline{L}$ and $\frac{\dd}{\dd t}[\overline{F}(-V(t))] \ge 0$ follows with
\[ \overline{F}(-V):=\sqrt{-2\overline{L}(-V)}=F(V). \]
Therefore, $F(V(t))$ is constant and $V(t)$ is timelike for all $t \in [0,1]$.
\end{proof}

\begin{proposition}[{cf.\ \cite[Proposition 4.1]{IL}}]
\label{pr:IL4.1}
Let $(M,L)$ be a Finsler spacetime,
and suppose that there is a Lorentzian metric $g$ on $M$ which is preserved
by every parallel transport along any geodesic.
Then, $g$ and $L$ have the same geodesics and, in particular, $L$ is of Berwald type.
\end{proposition}

\begin{proof}
In this proof, we emphasize ``$g$-'' for objects considered with respect to $g$; otherwise, they are understood with respect to $L$. 

Denote by $\nabla^g$ the $g$-covariant derivative.
Fix $x \in M$ and, for $v \in T_xM$, we define
\[ \kappa(v) :=\nabla^g_t [\dot{\eta}_v](0) \in T_xM, \]
where $\eta_v$ is the geodesic with $\dot{\eta}_v(0)=v$.
We shall show that $\kappa$ vanishes for all $v \in T_xM$, namely $\eta_v$ is also a $g$-geodesic.

Given $w \in T_xM$, let $W$ be the parallel vector field along $\eta_v$ such that $W(0)=w$
(i.e., $D_{\dot{\eta}_v}^{\dot{\eta}_v} W \equiv 0$).
We claim that
\begin{equation}\label{eq:IL4.1}
\frac{\dd}{\dd s} \Bigl[ \kappa(v+sw) \Bigr]_{s=0} =2 \nabla^g_t W(0). 
\end{equation}
The left-hand side can be understood in terms of local coordinates as
\[ \sum_{i=1}^n \frac{\dd}{\dd s} \Bigl[ \kappa^i(v+sw) \Bigr]_{s=0} \cdot \frac{\del}{\del x^i}\bigg|_x \in T_xM. \]
To this end, we employ a normal coordinate system with respect to $g$ around $x$
(see, e.g., \cite[Proposition~3.33]{ON}).
Then, the $g$-Christoffel symbols at $x$ vanish, thereby
\[ \kappa(v) =\sum_{i=1}^n \ddot{\eta}_v^i(0) \frac{\del}{\del x^i}\bigg|_x, \qquad
 \nabla^g_t W(0) =\sum_{i=1}^n \dot{W}^i(0) \frac{\del}{\del x^i}\bigg|_x. \]
This implies that
\[ \frac{\dd}{\dd s} \Bigl[ \kappa(v+sw) \Bigr]_{s=0}
 =\sum_{i=1}^n \frac{\dd}{\dd s} \Bigl[ \ddot{\eta}^i_{v+sw}(0) \Bigr]_{s=0} \cdot \frac{\del}{\del x^i}\bigg|_x
 =\sum_{i=1}^n \frac{\dd^2}{\dd t^2} \biggl[ \frac{\partial}{\partial s} \Bigl[ \eta^i_{v+sw}(t) \Bigr]_{s=0} \biggr]_{t=0} \cdot \frac{\del}{\del x^i}\bigg|_x. \]
We set
\[ J(t) :=\sum_{i=1}^n \frac{\partial}{\partial s} \Bigl[ \eta^i_{v+sw}(t) \Bigr]_{s=0} \cdot \frac{\del}{\del x^i}\bigg|_x, \]
which is a Jacobi field along $\eta_v$ such that $J(0)=0$ and $D^{\dot{\eta}_v}_{\dot{\eta}_v} J(0)=w$.
It follows that $J(t)=tW(t)+O(t^3)$ (as in the proof of Lemma~\ref{lm:IL3.1}), and hence
\[ \sum_{i=1}^n \ddot{J}^i(0) \frac{\del}{\del x^i}\bigg|_x
 =\sum_{i=1}^n 2\dot{W}^i(0) \frac{\del}{\del x^i}\bigg|_x
 =2\nabla^g_t W(0). \]
This yields \eqref{eq:IL4.1}.

Now, for any parallel vector fields $V,W$ along $\eta_v$, we infer from the hypothesis that
\begin{equation}\label{eq:IL4.3}
0 =\frac{\dd}{\dd t} \Bigl[ g(V,W) \Bigr] =g(\nabla^g_t V,W) +g(V,\nabla^g_t W).
\end{equation}
On the one hand, plugging $V=W=\dot{\eta}_{v+sw}$ into \eqref{eq:IL4.3} at $t=0$ implies $g(\kappa(v+sw),v+sw)=0$.
Combining this with \eqref{eq:IL4.1}, we find, for $W$ as in \eqref{eq:IL4.1},
\[ 0 =\frac{\dd}{\dd s}\Bigl[ g\bigl( \kappa(v+sw),v+sw \bigr) \Bigr]_{s=0}
 =2g\bigl( \nabla^g_t W(0),v \bigr) +g\bigl( \kappa(v),w \bigr). \]
On the other hand, \eqref{eq:IL4.3} with $V=\dot{\eta}_v$ and $W$ as in \eqref{eq:IL4.1} at $t=0$ shows
\[ g\bigl( \kappa(v),w \bigr) +g\bigl (v,\nabla^g_t W(0) \bigr) =0. \]
Comparing these equations, we obtain $g(\kappa(v),w)=0$.
Since $w \in T_xM$ was arbitrary, $\kappa(v)=0$ holds.
Therefore, $g$ and $L$ have the same geodesics.
This implies that the Chern connection of $L$ coincides with the Levi-Civita connection of $g$, thereby $L$ is of Berwald type.
\end{proof}

\begin{remark}[Canonical Lorentzian metrics]\label{rm:cano}
In order to build a bridge between Lemma~\ref{lm:IL3.1} and Proposition~\ref{pr:IL4.1},
in the situation of Lemma~\ref{lm:IL3.1}, we need to construct a Lorentzian metric $g$
satisfying the hypothesis in Proposition~\ref{pr:IL4.1}.
In the positive definite case, it was done in \cite[Proposition~4.2]{IL} by using a Riemannian metric ``canonically'' associated with a Finsler metric.
Such a construction is not known in Lorentz--Finsler geometry;
this is indeed a major obstacle in the metrization problem (recall Remark~\ref{rm:met}).
It is sufficient to find a ``canonical'' way of choosing a finite measure on each tangent space $T_xM$,
as taking an average of $g_v$ in $v \in T_xM$ with that measure provides a Lorentzian metric
(in the positive definite case, the uniform measure on the unit ball or sphere plays a role).
However, the noncompactness of the isometry group of Minkowski metrics arises
as an essential difficulty in such a construction.
\end{remark}

\subsection{Concavity and the variance functional}\label{ssc:var}

In the positive definite case, there is an interesting nonsmooth characterization of Busemann's convexity in terms of optimal transport theory.
We refer to \cite{Vi} for the basic knowledge of optimal transport theory.

Let $(X,d)$ be a complete geodesic space.
Given a Borel probability measure $\mu \in \cP_2(X)$ of finite second moment, its \emph{variance} is defined by
\[ \var(\mu) :=\inf_{x \in X} \int_X d^2(x,y) \,\mu(\dd y). \]
A point $x \in X$ achieving the above infimum is called a \emph{barycenter} (also called \emph{center of mass} or \emph{Fr\'echet mean}) of $\mu$.
If $(X,d)$ is globally convex, then the set of barycenters of $\mu$ is a convex set.
In the case where $(X,d)$ is a complete CAT$(0)$-space, any $\mu$ admits a unique barycenter. 

\begin{proposition}[Kim--Pass~\cite{KP}]\label{pr:var}
A complete separable geodesic space $(X,d)$ is globally convex if and only if $\sqrt{\var}$ is convex along geodesics in the $L^2$-Wasserstein space $(\cP_2(X),W_2)$.
\end{proposition}

\begin{proof}
To be precise, it was shown in \cite[Proposition~2.1]{KP} that $d^2$ is convex if and only if $\var$ is convex.
For completeness, we give a proof along the same lines as \cite{KP}.

In the ``if'' part, given minimizing geodesics $\gamma_1,\gamma_2 \colon [0,1] \lra X$, with $\gamma_1(0)=\gamma_2(0)$,
consider probability measures $\mu_t:=(\delta_{\gamma_1(t)}+\delta_{\gamma_2(t)})/2$.
Then, $(\mu_t)_{t \in [0,1]}$ is a $W_2$-geodesic and
\[ \var(\mu_t) =\frac{1}{4}d^2 \bigl( \gamma_1(t),\gamma_2(t) \bigr) \]
for all $t \in [0,1]$ (any midpoint of $\gamma_1(t)$ and $\gamma_2(t)$ is a barycenter of $\mu_t$).
Hence, the convexity of $\sqrt{\var(\mu_t)}$ implies that $t \longmapsto d(\gamma_1(t),\gamma_2(t))$ is convex.

For general minimizing geodesics $\gamma_1,\gamma_2 \colon [0,1] \lra X$, let $\eta\colon [0,1] \lra X$ be a minimizing geodesic from $\gamma_1(0)$ to $\gamma_2(1)$.
Then, we deduce from the triangle inequality and the above estimate that
\begin{align*}
d\bigl( \gamma_1(t),\gamma_2(t) \bigr)
&\le d\bigl( \gamma_1(t),\eta(t) \bigr) +d\bigl( \eta(t),\gamma_2(t) \bigr) \\
&\le td\bigl( \gamma_1(1),\gamma_2(1) \bigr) +(1-t)d\bigl( \gamma_1(0),\gamma_2(0) \bigr).
\end{align*}
Thus, $(X,d)$ is globally convex.

To see the ``only if'' part, let $\mu_t =(e_t)_* \Pi$ be a $W_2$-geodesic induced from a probability measure $\Pi$ on the set $\Gamma(X)$ of minimizing geodesics $\eta\colon [0,1] \lra X$,
where $e_t(\eta):=\eta(t)$ and $(e_t)_* \Pi$ denotes the push-forward of $\Pi$ by $e_t$
(see \cite[Theorem~6]{Li}).
For arbitrary $\ve>0$, take $x_i \in X$ such that
\[ \biggl( \int_X d^2(x_i,y) \,\mu_i(\dd y) \biggr)^{1/2} \le \sqrt{\var(\mu_i)} +\ve, \qquad i=0,1. \]
Let $\gamma\colon [0,1] \lra X$ be the (unique) minimizing geodesic from $x_0$ to $x_1$.
Then, it follows from the global convexity of $(X,d)$ that
\begin{align*}
\sqrt{\var(\mu_t)} &\le \biggl( \int_X d^2\bigl( \gamma(t),y \bigr) \,\mu_t(\dd y) \biggr)^{1/2}
 =\biggl( \int_{\Gamma(X)} d^2\bigl( \gamma(t),\eta(t) \bigr) \,\Pi(\dd\eta) \biggr)^{1/2} \\
&\le \biggl( \int_{\Gamma(X)} \Bigl( (1-t)d\bigl( \gamma(0),\eta(0) \bigr) +td\bigl( \gamma(1),\eta(1) \bigr) \Bigr)^2 \,\Pi(\dd\eta) \biggr)^{1/2} \\
&\le (1-t) \biggl( \int_{\Gamma(X)} d^2\bigl( \gamma(0),\eta(0) \bigr) \,\Pi(\dd\eta) \biggr)^{1/2}
 +t \biggl( \int_{\Gamma(X)} d^2\bigl( \gamma(1),\eta(1) \bigr) \,\Pi(\dd\eta) \biggr)^{1/2} \\
&= (1-t) \biggl( \int_X d^2(x_0,y) \,\mu_0(\dd y) \biggr)^{1/2} +t \biggl( \int_X d^2(x_1,y) \,\mu_1(\dd y) \biggr)^{1/2} \\
&\le (1-t)\sqrt{\var(\mu_0)} +t\sqrt{\var(\mu_1)} +\ve.
\end{align*}
Letting $\ve \to 0$ shows that $t \longmapsto \sqrt{\var(\mu_t)}$ is convex.
\end{proof}

\begin{remark}[Variance and barycenter]\label{rm:var}
To formulate a Lorentzian counterpart to Proposition~\ref{pr:var}, we need an appropriate notion of variance, which is an open question.
This is related to another open question: how to define the barycenter of a probability measure in a Lorentzian or Finsler spacetime?
Even in the Minkowski space $(\R^n,L)$, no reasonable characterization of the linear average $\int_{\R^n} x \,\mu(\dd x)$ in a synthetic/metric way seems to be known.
These notions will be helpful to develop probability theory and statistics on spacetimes.
\end{remark}

\paragraph{Acknowledgements}
This research was supported in part by the Austrian Science Fund (FWF) [Grants DOI \href{https://doi.org/10.55776/PAT1996423}{10.55776/PAT1996423} and \href{https://doi.org/10.55776/EFP6}{10.55776/EFP6}].
SO was supported by the JSPS Grant-in-Aid for Scientific Research (KAKENHI) 22H04942, 24K00523.
FR acknowledges the support of the European Union - NextGenerationEU, in the framework of the PRIN Project `Contemporary perspectives on geometry and gravity' (code 2022JJ8KER – CUP G53D23001810006). The views and opinions expressed are solely those of the authors and do not necessarily reflect those of the European Union, nor can the European Union be held responsible for them.

For open access purposes, the authors have applied a \href{https://creativecommons.org/licenses/by/4.0/}{Creative Commons Attribution 4.0 International} license to any author-accepted manuscript version arising from this submission.

\end{document}